\documentclass[12pt,dvipsnames,a4paper]{article}
\usepackage[affil-it]{authblk}
\usepackage[utf8]{inputenc}
\usepackage{amsmath, amssymb, mathrsfs, amsthm, amsfonts}
\usepackage{latexsym}
\usepackage{hyperref}
\usepackage{a4wide}
\usepackage{color,xcolor}
\usepackage{graphicx}
\usepackage{fancyhdr}

\usepackage{natbib}

\newtheorem{theorem}{Theorem}[section]

\newtheorem{lemma}[theorem]{Lemma}
\newtheorem{proposition}[theorem]{Proposition}

\theoremstyle{definition}

\newtheorem{algorithm}{Algorithm}
\theoremstyle{remark}
\newtheorem*{example}{Example}

\newcommand{\precdot}{\prec\mathrel{\mkern-5mu}\mathrel{\cdot}}

\begin{document}

\thispagestyle{empty}

\begin{center}
{\color{MidnightBlue} \sffamily \Huge A Generalization of a Theorem of Mandel}
\end{center}

\vspace*{20pt}

\noindent \textbf{\Large Hery Randriamaro} 

\vspace*{10pt}

\noindent \textsc{Institut für Mathematik, Universität Kassel,\\ Heinrich-Plett-Straße 40,\\ 34132 Kassel, Germany} \\ \texttt{hery.randriamaro@mathematik.uni-kassel.de} \\ The author was supported by the Alexander von Humboldt Foundation

\vspace*{20pt}

\noindent \textsc{\large Abstract} 

\noindent A theorem of Mandel allows to determine the covector set of an oriented matroid from its set of topes by using the composition condition. We provide a generalization of that result, stating that the covector set of a conditional oriented matroid can also be determined by its set of topes, but by using the face symmetry condition. It permits to represent geometrical configurations in terms of conditional oriented matroids, more suitable for computer calculations. We treat apartments of hyperplane arrangements as example.

\vspace*{10pt}


\vspace*{10pt}


\vspace*{10pt}

\noindent {\color{MidnightBlue} \rule{\linewidth}{2pt}}

\vspace*{15pt}

\noindent Recently, \cite{bandelt2018coms} introduced the notion of conditional oriented matroids, or complexes of oriented matroids, which are common generalizations of oriented matroids and lopsided sets. As observed by \cite{richter20176}, oriented matroids are abstractions for several mathematical objects including directed graphs and central hyperplane arrangements, while \cite{bandelt2006combinatorics} pointed out that lopsided sets can be regarded as common generalizations of antimatroids and median graphs. We provide a generalization of a theorem of Mandel in Section~\ref{SeCOM} by proving that a conditional oriented matroid can completely determined from knowledge of its tope and by means of the face symmetry condition. \cite{knauer2020tope}, independently without the author having prior knowledge, gave another version of that generalization in their Theorem~4.9 by using tope graphs. In Section~\ref{SeApp}, we propose an algorithm to convert apartments of hyperplane arrangements to conditional oriented matroids. It gives the possibility to do computations, like $f$-polynomial computing, on these geometrical configurations.

\section{Conditional Oriented Matroids} \label{SeCOM}

\noindent This section describes conditional oriented matroids, recalls oriented matroids and the theorem of Mandel, then establishes our generalization of that theorem.

\smallskip

\noindent A \emph{sign system} is a pair $(E,\,\mathcal{L})$ containing a finite set $E$ and a subset $\mathcal{L}$ of $\{-1,\,0,\,1\}^E$. For $X,Y \in \mathcal{L}$, the \emph{composition} of $X$ and $Y$ is the element $X \circ Y$ of $\{-1,\,0,\,1\}^E$ defined, for every $e \in E$, by
$$(X \circ Y)_e := \begin{cases} X_e & \text{if}\ X_e \neq 0, \\ Y_e & \text{otherwise}, \end{cases}$$
and the \emph{separation set} of $X$ and $Y$ is $\mathrm{S}(X,\,Y) := \big\{e \in E\ \big|\ X_e = -Y_e \neq 0\big\}$. A \emph{conditional oriented matroid} is a sign system $(E,\,\mathcal{L})$ such that $\mathcal{L}$ satisfies the following conditions:
\vspace*{7pt}
\begin{itemize}
\item[(\textbf{FS})] if $X,Y \in \mathcal{L}$, then $X \circ -Y \in \mathcal{L}$,
\item[(\textbf{SE})] for each pair $X,Y \in \mathcal{L}$, and every $e \in \mathrm{S}(X,\,Y)$, there exists $Z \in \mathcal{L}$ such that
$$Z_e = 0 \quad \text{and} \quad \forall f \in E \setminus \mathrm{S}(X,\,Y),\ Z_f = (X \circ Y)_f = (Y \circ X)_f.$$
\end{itemize}
(FS) stands for face symmetry, and (SE) for strong elimination condition.

\noindent The elements of $\mathcal{L}$ are called \emph{covectors}. A partial order $\preceq$ is defined on $\mathcal{L}$ by $$\forall X, Y \in \mathcal{L}:\ X \preceq Y \ \Longleftrightarrow \ \forall e \in E,\, X_e \in \{0,\, Y_e\}.$$ Write $X \prec Y$ if $X \preceq Y$ and $X \neq Y$. One says that $Y$ \emph{covers} $X$, denoted $X \precdot Y$, if $X \prec Y$ and no $Z \in \mathcal{L}$ satisfies $X \prec Z \prec Y$. One calls $Y$ a \emph{tope} if no $Z \in \mathcal{L}$ covers $Y$. 

\smallskip

\noindent An \emph{oriented matroid} is a sign system $(E,\,\mathcal{L})$ such that $\mathcal{L}$ satisfies the following conditions:
\begin{itemize}
	\item[(\textbf{C})] if $X,Y \in \mathcal{L}$, then $X \circ Y \in \mathcal{L}$,
	\item[(\textbf{Sym})] if $X \in \mathcal{L}$, then $-X \in \mathcal{L}$,
	\item[(\textbf{SE})] for each pair $X,Y \in \mathcal{L}$, and every $e \in \mathrm{S}(X,\,Y)$, there exists $Z \in \mathcal{L}$ such that
	$$Z_e = 0 \quad \text{and} \quad \forall f \in E \setminus \mathrm{S}(X,\,Y),\ Z_f = (X \circ Y)_f = (Y \circ X)_f.$$
\end{itemize}
(C) stands for composition, and (Sym) for symmetry condition. For the sake of understanding, we provide a proof to the following known property.

\begin{proposition} \label{PrOM}
An oriented matroid is a conditional oriented matroid $(E,\,\mathcal{L})$ that satisfies the zero vector condition
\begin{itemize}
		\item[(\textbf{Z})]	the zero element $(0,\, \dots,\, 0)$ belongs to $\mathcal{L}$.
	\end{itemize}
\end{proposition}

\begin{proof}
Every oriented matroid is a conditional oriented matroid since both satisfy (SE), and one gets (FS) by combining (C) with (Sym). Moreover, it is obvious that every oriented matroid satisfies (Z). 
	
\smallskip
	
\noindent Now, suppose that $(E,\,\mathcal{L})$ is a conditional oriented matroid satisfying (Z):
\begin{itemize}
\item For every $X \in \mathcal{L}$, $(0,\, \dots,\, 0) \circ -X = -X \in \mathcal{L}$, so we get (Sym).
\item If $X,Y \in \mathcal{L}$, then $-Y \in \mathcal{L}$, hence $X \circ -(-Y) = X \circ Y \in \mathcal{L}$, so we get (C).
\end{itemize}
\end{proof}

\noindent We recall the Theorem of Mandel as stated in Theorem~4.2.13 of \cite{bjorner1999oriented}. One can look at Theorem~1.1 of \cite{cordovil1985combinatorial} for a version using non-Radon partitions.

\begin{theorem}
Let $(E,\,\mathcal{L})$ be an oriented matroid. Its set of topes $\mathcal{T}$ determines $\mathcal{L}$ via $$\mathcal{L} = \big\{X \in \{-1,\,0,\,1\}^E\ \big|\ \forall T \in \mathcal{T},\, X \circ T \in \mathcal{T}\big\}.$$
\end{theorem}

\noindent Coming back to conditional oriented matroids, the \emph{rank} of a covector $X$ is $0$ if it covers no elements in $\mathcal{L}$, otherwise it is
$$\mathrm{rk}\,X := \max \{l \in \mathbb{N}\ |\ \exists X^1, X^2, \dots, X^l \in \mathcal{L},\, X^1 \precdot X^2 \precdot \dots \precdot X^l \precdot X\}.$$
The rank of $\mathcal{L}$ is $$\mathrm{rk}\,\mathcal{L} := \max \{\mathrm{rk}\,X \ |\ X \in \mathcal{L}\}.$$
The \emph{support} of $X$ is $\underline{X} := \{e \in E\ |\ X_e \neq 0\}$. And for $A \subseteq E$, the \emph{restriction} of $X$ to $E \setminus A$ is the element $X \setminus A \in \{-1,\,0,\,1\}^{E \setminus A}$ such that $(X \setminus A)_e = X_e$ for all $e \in E \setminus A$.

\begin{lemma}\citep[Lem.~1]{bandelt2018coms}
Let $(E,\,\mathcal{L})$ be a conditional oriented matroid, and $A \subseteq E$.
\begin{itemize}
\item The \emph{deletion} $(E \setminus A,\, \mathcal{L} \setminus A)$ of $A$, with $\mathcal{L} \setminus A = \{X \setminus A\ |\ X \in \mathcal{L}\}$,
\item and the \emph{contraction} $(E \setminus A,\, \mathcal{L}/A)$ of $A$, with $\mathcal{L}/A = \{X \setminus A\ |\ X \in \mathcal{L},\, \underline{X} \cap A = \varnothing\}$,
\end{itemize}
are conditional oriented matroids.
\end{lemma}

\begin{lemma}
Let $(E,\,\mathcal{L})$ be a conditional oriented matroid, and take two topes $T^1, T^2 \in \mathcal{L}$. Then, $\underline{T^1} = \underline{T^2}$.
\end{lemma}

\begin{proof}
Suppose that $\underline{T^1} \neq \underline{T^2}$. Then, $T^1 \circ -T^2 \in \mathcal{L}$ and $T^1 \prec T^1 \circ -T^2$. This implies that $T^1$ is not a tope, which is absurd.
\end{proof}

\noindent We can now state our generalization.

\begin{theorem} \label{ThMandel}
Let $(E,\,\mathcal{L})$ be a conditional oriented matroid. Its set of topes $\mathcal{T}$ determines $\mathcal{L}$ via $$\mathcal{L} = \big\{X \in \{-1,\,0,\,1\}^E\ \big|\ \forall T \in \mathcal{T},\, X \circ -T \in \mathcal{T}\big\}.$$
\end{theorem}

\begin{proof}
It is clear that $\mathcal{L} \subseteq \big\{X \in \{-1,\,0,\,1\}^E\ \big|\ \forall T \in \mathcal{T},\, X \circ T \in \mathcal{T}\big\}$ since, for every $X \in \mathcal{L}$ and all $T \in \mathcal{T}$, $X \circ -T \in \mathcal{L}$ and $(X \circ T)^0 = \varnothing$.
	
\bigskip
	
\noindent For the backward argument, we argue by induction on $\mathrm{rk}\,\mathcal{L}$ and $\#E$. If $\mathrm{rk}\,\mathcal{L} = 0$, $\mathcal{L}$ consists of a one-element set $\{T\} \subseteq \{-1,\,0,\,1\}^E$. Therefore, for an element $X \in \{-1,\,0,\,1\}^E$, the fact $X \circ -T = T$ implies $X = T$, hence $X \in \mathcal{L}$.
	
\smallskip
	
\noindent If $\mathrm{rk}\,\mathcal{L} = 1$ and $\#E = 1$, then $\mathcal{L} = \{-1,\,0,\,1\}$ and $\mathcal{T} = \{-1,\, 1\}$. So, we clearly have $X \in \mathcal{L}$ for all $X \in \{-1,\,0,\,1\}$.
	
\smallskip
	
\noindent Now, assume that $\mathrm{rk}\,\mathcal{L} = 1$ and $\#E > 1$. Take $X \in \{-1,\,0,\,1\}^E$ such that $X \circ -T \in \mathcal{T}$ for each tope $T \in \mathcal{T}$. Denote by $F$ the subset of $E$ such that $\underline{T} = F$ for every $T \in \mathcal{T}$. The case $\underline{X} = F$ is easily solved, since $X = X \circ -T \in \mathcal{T} \subseteq \mathcal{L}$. The case $\underline{X} \varsubsetneq F$ remains open. Pick an element $e \in X^0 \cap F$, and consider a tope $Y \setminus \{e\}$ of the deletion $(E \setminus \{e\},\, \mathcal{L} \setminus \{e\})$. We have $Y^0 = \{e\} \cap F$, and $Y$ is covered in $\mathcal{L}$ by two topes $T^1, T^2 \in \mathcal{T}$ such that $\mathrm{S}(T^1,\,T^2) = \{e\}$. There exists $Z \in \mathcal{L}$ such that
$$Z_e = 0 \quad \text{and} \quad \forall f \in E \setminus \{e\},\ Z_f = \big((X \circ -T^1) \circ (X \circ -T^2)\big)_f = (X \circ -Y)_f.$$
The only possibility is $Z = X \circ -Y$, which means that $X \circ -Y \in \mathcal{L}$. Hence, for all topes $Y \setminus \{e\}$ in the deletion $(E \setminus \{e\},\, \mathcal{L} \setminus \{e\})$, we have $X \setminus \{e\} \circ -(Y \setminus \{e\}) \in \mathcal{L} \setminus \{e\}$. By induction, we get $X \setminus \{e\} \in \mathcal{L} \setminus \{e\}$, and consequently $X \in \mathcal{L}$.
	
\smallskip
	
\noindent Finally, assume that $\mathrm{rk}\,\mathcal{L} > 1$. Take $X \in \{-1,\,0,\,1\}^E$ such that $X \circ -T \in \mathcal{T}$ for each tope $T \in \mathcal{T}$. The case $\underline{X} = F$ is easily solved like before. The case $\underline{X} \varsubsetneq F$ remains. Pick an element $e \in X^0 \cap F$, and consider a tope $Y \setminus \{e\}$ of the contraction $(E \setminus \{e\},\, \mathcal{L}/\{e\})$. We have $Y^0 = \{e\} \cap F$, and $Y$ is covered in $\mathcal{L}$ by two topes $T^1, T^2 \in \mathcal{T}$ such that $\mathrm{S}(T^1,\,T^2) = \{e\}$. There exists $Z \in \mathcal{L}$ such that
$$Z_e = 0 \quad \text{and} \quad \forall f \in E \setminus \{e\},\ Z_f = \big((X \circ -T^1) \circ (X \circ -T^2)\big)_f = (X \circ -Y)_f.$$
The only possibility is $Z = X \circ -Y$, which means that $X \circ -Y \in \mathcal{L}$. Hence, for all topes $Y \setminus \{e\}$ in the contraction $(E \setminus \{e\},\, \mathcal{L}/\{e\})$, we have $X \setminus \{e\} \circ -(Y \setminus \{e\}) \in \mathcal{L}/\{e\}$. Since $\mathrm{rk}\,\mathcal{L}/\{e\} = \mathrm{rk}\,\mathcal{L} - 1$, then $X \setminus \{e\} \in \mathcal{L}/\{e\}$ by induction, and consequently $X \in \mathcal{L}$.
\end{proof}

\section{Applications on Hyperplane Arrangements} \label{SeApp}

\noindent This section describes the structure of apartments of hyperplane arrangements in term of conditional oriented matroids. Then, it proposes an algorithm to convert the former to the latter. We give the $f$-polynomial computing as extension example of this algorithm.  

\smallskip

\noindent Let $a_1, \dots, a_n, b$ be $n+1$ real coefficients such that $(a_1,\,\dots,\,a_n) \neq (0,\,\dots,\,0)$. A \emph{hyperplane} of $\mathbb{R}^n$ is an affine subspace $H := \big\{(x_1,\, \dots,\, x_n) \in \mathbb{R}^n \ \big|\ a_1x_1 + \dots + a_nx_n = b\big\}$ denoted by $\{a_1x_1 + \dots + a_nx_n = b\}$. A \emph{hyperplane arrangement} $\mathscr{A}$ is a finite set of hyperplanes. Denote by $H^{-1}$ and $H^1$ both connected components $\{a_1x_1 + \dots + a_nx_n < b\}$ and  $\{a_1x_1 + \dots + a_nx_n > b\}$ of $\mathbb{R}^n$, respectively. Moreover, set $H^0 = H$. The \emph{sign map} of $H$ is the function
$$\sigma_H: \mathbb{R}^n \rightarrow \{-1,\,0,\,1\}, \quad v \mapsto \begin{cases}
	-1 & \text{if}\ v \in H^{-1}, \\
	0 & \text{if}\ v \in H^0, \\
	1 & \text{if}\ v \in H^1. 
\end{cases}$$
The sign map of $\mathscr{A}$ is the function $\sigma_{\mathscr{A}}: \mathbb{R}^n \rightarrow \{-1,\,0,\,1\}^{\mathscr{A}}, \ v \mapsto \big(\sigma_H(v)\big)_{H \in \mathscr{A}}$. And the \emph{sign set} of $\mathscr{A}$ is the set $\sigma_{\mathscr{A}}(\mathbb{R}^n) := \big\{\sigma_{\mathscr{A}}(v)\ \big|\ v \in \mathbb{R}^n\big\}$. A \emph{face} of $\mathscr{A}$ is a subset $F$ of $\mathbb{R}^n$ such that
$$\exists x \in \sigma_{\mathscr{A}}(\mathbb{R}^n),\ F = \big\{v \in \mathbb{R}^n\ \big|\ \sigma_{\mathscr{A}}(v) = x\big\}.$$
A \emph{chamber} of $\mathscr{A}$ is a face $F$ such that $\sigma_{\mathscr{A}}(F) \in \{-1,\,1\}^{\mathscr{A}}$. Denote by $F(\mathscr{A})$ and $C(\mathscr{A})$ the sets composed by the faces and the chambers of $\mathscr{A}$, respectively. An \emph{apartment} of $\mathscr{A}$ is a chamber of a hyperplane arrangement contained in $\mathscr{A}$. Denote by $K(\mathscr{A})$ the apartment set of $\mathscr{A}$. The sets of faces and chambers in an apartment $K \in K(\mathscr{A})$ are, respectively, $$F(\mathscr{A},\,K) := \big\{F \in F(\mathscr{A})\ |\ F \subseteq K\big\} \quad \text{and} \quad C(\mathscr{A},\,K) := C(\mathscr{A}) \cap F(\mathscr{A},\,K).$$ 

\noindent Let $K \in K(\mathscr{A})$, and $\mathscr{B} = \{H \in \mathscr{A}\ |\ H \cap K = \varnothing\}$. The sign system $$\Big(\mathscr{A} \setminus \mathscr{B},\, \sigma_{\mathscr{A}}\big(F(\mathscr{A},\,K)\big) \setminus \mathscr{B}\Big)$$ is a conditional oriented matroid. \cite{bandelt2018coms} called it realizable COMs, and presented it as motivating example for conditional oriented matroids. 

\smallskip

\noindent We now present algorithms to do computations on apartments of hyperplane arrangements. The use of mathematics software system containing the following functions is assumed:
\begin{itemize}
\item \textbf{length} gives the length of a tuple,
\item \textbf{RandomElement} returns randomly an element from a set,
\item \textbf{poset} transforms a set, on which a partial order can be defined, to a poset,
\item and \textbf{rank} computes the rank of a poset or that of its elements. 
\end{itemize}

\begin{algorithm}(Generating Conditional Oriented Matroid from Topes):
	\begin{itemize}
		\item \textit{Input:} A tope set $\mathcal{T}$.
		\item \textit{Output:} A covector set $\mathcal{L}$.
		\item \textit{Remark:} It is an algorithmic version of Theorem~\ref{ThMandel}.
	\end{itemize}
	\textbf{function} GeneratingCOM(T) \\
	\hspace*{15pt} L $\leftarrow$ $\{\}$ \\
	\hspace*{15pt} l $\leftarrow$ \textbf{length}(\textbf{RandomElement}(T)) \\
	\hspace*{15pt} \textbf{for} X \textbf{in} $\{-1,\,0,\,1\}^l$ \\
	\hspace*{30pt} a $\leftarrow$ true \\
	\hspace*{30pt} \textbf{for} Y \textbf{in} T \\
	\hspace*{45pt} a $\leftarrow$ a \text{and} (X $\circ$ -Y \textbf{in} T) \\
	\hspace*{30pt} \textbf{if} a = true \\
	\hspace*{45pt} L $\leftarrow$ L $\sqcup$ $\{$X$\}$ \\
	\hspace*{15pt} \textbf{return} L
\end{algorithm}

\noindent We generate tope set by determining $\sigma_{\mathscr{A}}(v)$ for a random point of each chamber. Afterwards, we apply the previous algorithm to get the aimed conditional oriented matroid.

\begin{algorithm}(Transforming Apartment to Conditional Oriented Matroid):
	\begin{itemize}
		\item \textit{Input:} An affine function set $\mathscr{A}$ and a point set $P$.
		\item \textit{Output:} A covector set $\mathcal{L}$.
		\item \textit{Remark:} Each function in $\mathscr{A}$ corresponds to a hyperplane of the arrangement, and each point in $P$ is included in a chamber of the arrangement. 
	\end{itemize}
	\noindent \textbf{function} ApartmentToCOM(A, P) \\
	\hspace*{15pt} \textbf{function} ApartmentToTope(A, P) \\
	\hspace*{30pt} \textbf{function} covector(A, p) \\
	\hspace*{45pt} \textbf{function} sign(h, p) \\
	\hspace*{60pt} \textbf{if} h(p) $<$ 0 \\
	\hspace*{75pt} \textbf{return} -1 \\
	\hspace*{60pt} \textbf{else} \\
	\hspace*{75pt} \textbf{return} 1 \\
	\hspace*{45pt} \textbf{return} \textbf{tuple}(\textbf{sign}(h, p) \textbf{for} h \textbf{in} A) \\
	\hspace*{30pt} \textbf{return} \textbf{set}(\textbf{covector}(A, p) \textbf{for} p \textbf{in} P) \\
	\hspace*{15pt} \textbf{return} \textbf{GeneratingCOM}(ApartmentToTope(A, P))
\end{algorithm}

\noindent Consider an apartment $K \in K(\mathscr{A})$ in $\mathbb{R}^n$. Let $f_i(K)$ be the number of $i$-dimensional faces in $F(\mathscr{A},\,K)$, and $x$ a variable. The \emph{$f$-polynomial} of $K$ is $$f_K(x) := \sum_{i=0}^n f_i(K)\, x^{n-i}.$$

\begin{algorithm}($f$-Polynomial of Apartment):
\begin{itemize}
\item \textit{Input:} An affine function set $\mathscr{A}$ and a point set $P$.
\item \textit{Output:} A $f$-polynomial $f_K(x)$.
\item \textit{Remark:} An apartment is still represented by a pair $(\mathscr{A},\,E)$.
\end{itemize}
\textbf{function} fPolynomial(A, P) \\
\hspace*{15pt} $x$ \textbf{variable} \\
\hspace*{15pt} f $\leftarrow$ 0 \\
\hspace*{15pt} COM $\leftarrow$ \textbf{poset}(\textbf{ApartmentToCOM}(A, P)) \\
\hspace*{15pt} \textbf{for} X \textbf{in} COM \\
\hspace*{30pt} f $\leftarrow$ f + $x^{\mathbf{rank}(COM) - \mathbf{rank}(X)}$ \\
\hspace*{15pt} \textbf{return} f
\end{algorithm}

\noindent Other computer calculations of functions associated to apartments of hyperplane arrangements, like their Varchenko determinant \citep[Th.~1.3]{randriamaro2020varchenko}, can also be implemented by means of their conversion to conditional oriented matroids.

\begin{example}
Consider the apartment on Figure~\ref{Pic} with arrangement composed by the hyperplanes $\{x_2 = 0\}$, $\{x_1 - x_2 = 0\}$, $\{x_1 + x_2 = 1\}$, $\{x_2 = 3\}$, and $\{x_2 = -2\}$. To be able to compute the corresponding conditional oriented arrangement, and its $f$-polynomial, we take the nine points $(0,\,4)$,  $(0,\,1.5)$, $(0,\,0.5)$, $(0.5,\,0.2)$, $(1,\,0.2)$, $(-1,\,-0.2)$, $(0,\,-0.5)$, $(1.5,\,-0.2)$, and $(0,\,-3)$ into account. A computation with the mathematics software system \href{https://www.sagemath.org}{SageMath} gives us the following result.

\begin{figure}[h]
\centering
\includegraphics[scale=1]{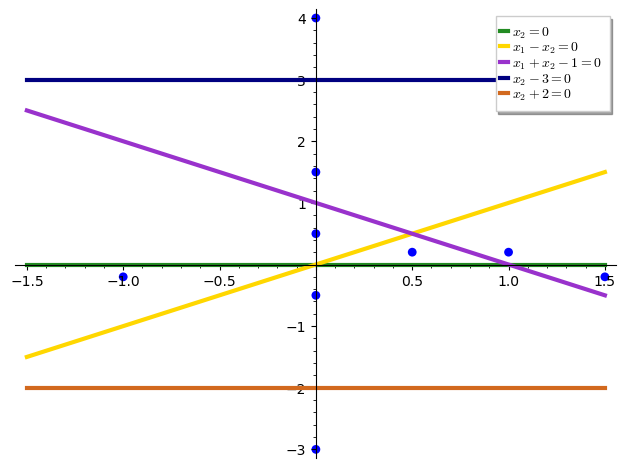}
\caption{An Apartment of a Hyperplane Arrangement}
\label{Pic}
\end{figure}

\begin{verbatim}
sage: ApartmentToCOM(A, P)
{(1, -1, -1, -1, 1), (1, 0, -1, -1, 1), (1, 1, 0, -1, 1), (0, 1, 1, -1, 1), 
(-1, 1, 0, -1, 1), (1, 0, 0, -1, 1), (1, -1, 1, 0, 1), (1, 1, 1, -1, 1),
(1, -1, 0, -1, 1), (0, 1, -1, -1, 1), (1, -1, 1, -1, 1), (1, 0, 1, -1, 1),
(-1, 1, 1, -1, 1), (-1, 1, -1, -1, 0), (0, -1, -1, -1, 1), (0, 0, -1, -1, 1),
(0, 1, 0, -1, 1), (1, -1, 1, 1, 1), (1, 1, -1, -1, 1), (-1, 0, -1, -1, 1),
(-1, 1, -1, -1, -1), (-1, -1, -1, -1, 1), (-1, 1, -1, -1, 1)}
sage: fPolynomial(A, P)
3*x^2 + 11*x + 9
\end{verbatim}
\end{example}

\bibliographystyle{agsm}
\bibliography{References}

\end{document}